\documentclass{amsart}
\usepackage{amsthm}
\usepackage{mathrsfs}
\usepackage{amsfonts}
\usepackage{amsmath}

\newtheorem{thm}{Theorem}[section]
\newtheorem{prp}[thm]{Proposition}
\newtheorem{lem}[thm]{Lemma}

\theoremstyle{definition}
\newtheorem{definition}[thm]{Definition}

\usepackage{mathrsfs}
\usepackage{amsfonts}

\DeclareMathOperator{\codim}{codim}
\DeclareMathOperator{\im}{im}
\DeclareMathOperator{\Pic}{Pic}

\begin{document}

\title{Smooth rational surfaces of $d=11$ and $\pi=8$ in $\mathbb{P}^5$.}

\author{Abdul Moeed Mohammad}
\address{Matematisk institutt, Universitetet i Oslo, PB 1053, Blindern, NO-0316 Oslo, Norway.}
\email{abdulmm@math.uio.no}



\date{April 5, 2013}

\begin{abstract}
We construct a linearly normal
smooth rational surface $S$ of degree $11$ and sectional genus $8$ in the projective fivespace. Surfaces
satisfying these numerical invariants are special, in the sense that $h^1(\mathscr{O}_S(1))>0$. Our construction
is done via linear systems and we describe the configuration of points blown up in the projective plane. Using the theory of adjunction mappings, we present a short list of linear systems whom are the only possibilities for other families of surfaces with
the prescribed numerical invariants.
\end{abstract}

\maketitle

\section{Introduction}

A result due to Arrondo, Sols and Pedreira \cite{4} states that there
are finitely many components of the Hilbert scheme corresponding to smooth surfaces of non-general type, i.e. Kodaira dimension $<2$, in $\mathbb{G}(1,3)$. Since the Pl$\ddot{\text{u}}$cker embedding embeds $\mathbb{G}(1,3)$ as a smooth quadric into $\mathbb{P}\left(\wedge^2 V\right)\simeq\mathbb{P}^5$, it follows that there are finitely many components of the Hilbert scheme corresponding to smooth surfaces of non-general type contained within a smooth quadric in $\mathbb{P}^5$. The latter idea led Papantonopoulou, Verra, Arrondo, Sols, and Gross to
classify smooth surfaces of non-general type with degree $\leq 10$ contained within smooth quadrics in $\mathbb{P}^5$. For a complete classification, see \cite{3} and \cite{4} for degree $\leq 9$  and \cite{10} for degree $10$.

A continuation of this classification is to study smooth surfaces of degree $11$ in $\mathbb{P}^5$ of non-general type. This paper focuses on a certain class of surfaces of non-general type, namely rational surfaces. Rational surfaces are birational to the projective plane $\mathbb{P}^2$ and all rational surfaces can be obtained by blowing up $\mathbb{P}^2$ or the Hirzebruch surfaces $\mathbb{F_e}$ in a finite number of points. Therefore, rational surfaces can be described by the geometry of the points blown up in $\mathbb{P}^2$ or $\mathbb{F}_e$.

In this paper, we study rational surfaces $S$ of degree $11$ blown up in points that are not in general position, so-called special rational surfaces, i.e. $h^1(\mathscr{O}_S(1))>0$. A simple application of Riemann-Roch yields that the sectional genus must be $\geq 8$. On the other hand, another application of Riemann-Roch tells us that if the sectional genus $>8$, then the hyperplane intersection lies on a quadric. We are interested in surfaces that do not interfere with \cite{10}. Therefore, we study smooth rational surfaces $S\subset\mathbb{P}^5$ of degree $11$ with sectional genus $8$ in this paper.

The study of rational surfaces is a classical topic of algebraic geometry, but was revived in the late 1980s due to a result by Ellingsrud and Peskine \cite{9} that states that finitely many components of the Hilbert scheme correspond to rational surfaces in $\mathbb{P}^4$. This resulted in a renewed interest for the classification of surfaces in $\mathbb{P}^4$, see the introductions of \cite{6} and \cite{12}, and led to the development of techniques for construction of rational surfaces. Three such techniques are Beilinson monad e.g. \cite{1},  linkage e.g. \cite{12} and complete linear systems e.g. \cite{8}.

Our point of view is the study of complete linear systems. The complete linear systems are obtained by some results on the adjunction mapping, due to Sommese and Van de Ven \cite{13}. Recall that every smooth surface can be embedded into $\mathbb{P}^5$ through generic projection. One consequence of this is that we have cannot expect similar relations between the invariants in $\mathbb{P}^5$ as one would in $\mathbb{P}^4$, so we obtain more complete linear systems then we would with in $\mathbb{P}^4$ through the adjunction mappings.

Our main result is the following:

\begin{thm}\label{mainT} There exists a family of linearly normal smooth rational surfaces in $\mathbb{P}^5$ with degree $11$ and sectional genus $8$. Each surface in the family is isomorphic to $\mathbb{P}^2$
blown up in $17$ points and the embedding complete linear system is 
\begin{eqnarray*}
|7L-2E_1-\hdots -2E_7-E_8-\hdots -E_{17}|
\end{eqnarray*}
where $L$ is the pullback of a line in $\mathbb{P}^2$ and $E_i$ are the exceptional curves of the blowup. Conversely, every linearly normal smooth rational
surface in $\mathbb{P}^5$ with degree $11$ and sectional genus $8$ has $$-10\leq K_S^2\leq -7$$ and the embedding complete linear systems is one of the following:\\
\\
\begin{tabular}{l l l}
(1). & $|7L-2E_1-\hdots -2E_7-E_8-\hdots -E_{17}|$ &\\
(2). & $|9L-3E_1-\hdots -3E_6-2E_7-2E_8-E_9-\hdots -E_{16}|$, &\\
(3). & $|4B+(4-2e)F-2E_1-E_2-\hdots -E_{18}|$, & where $e\leq 2$.\\
(4). & $|4B+(5-2e)F-2E_1-\hdots -2E_4-E_5-\hdots -E_{17}|$, & where $e\leq 3$.\\
(5). & $|4B+(6-2e)F-2E_1-\hdots -2E_7-E_8-\hdots -E_{16}|$, & where $e\leq 5$.
\end{tabular}\\
\\
where $B$ is a section with self-intersection $B^2=e$ on $\mathbb{F}_e$, $F$ is a ruling on $\mathbb{F}_e$, $L$ is the pullback of a line in $\mathbb{P}^2$ and $E_i$ are the exceptional curves of the blowup.
\end{thm}

In Section \ref{constt} we prove the existence of a smooth rational surface in $\mathbb{P}^5$ with the prescribed numerical invariants. The reader may refer to the proof of Theorem \ref{constEx1} for the details of the special configuration of the points blown up in $\mathbb{P}^2$. In \cite{4}, appendix to section 4, Arrondo and Sols outline an incomplete example of this surface without details. Our construction verifies their conjecture that there exists a quartic and sextic passing through the points blown up. Beside the construction of (1) in Theorem \ref{mainT}, Abo and Ranestad (unpublished) have used Macaulay2 to show that case (2) of Theorem \ref{mainT} is in the linkage class of a singular quadric surface and two smooth cubic surfaces in $\mathbb{P}^5$.

In Section \ref{liftSect} we give two examples of complete linear systems both of which cannot be both very ample and have six global sections
simultaneously, by lifting sections of curves on the surface to global sections on the surface. 

In Section \ref{adjSect} we obtain a finite list of complete linear systems
satisfying the numerical invariants of the surfaces in question using a result on adjoint linear systems, due to Sommese and Van de Ven.  Then we shorten the list to the five complete linear systems depicted in Theorem \ref{mainT} by making use of a result on curves of low degree, due to Catanese and Franciosi, to obtain contradictions and the lifting examples in Section \ref{liftSect}.

\section{Preliminaries}\label{preSec}

Throughout this paper we work over an algebraically closed field $k$ with characteristic $0$ and $S$ denotes a smooth rational projective surface over $k$. We use the following shorthand notation:\\
\begin{tabular}{l l l}
& &\\
$\pi_S$ & = & genus of a general hyperplane section of $S$.\\
$p_a(D)$ & = & arithmetic genus of a $D$ divisor on a curve or a surface.\\
$\tilde{\mathbb{P}}^2(x_1,\hdots, x_r)$ & = &  the projective plane blown up at the points $x_1,\hdots, x_r\in\mathbb{P}^2$.\\
$h^1(\mathscr{O}_X(D))$ & = & speciality of a divisor $D$ on a curve $X$ or a surface $X$ , i.e. $\dim H^1(X,\mathscr{O}_X(D))$.\\
$\mathbb{F}_e$ & = & Hirzebruch surfaces, i.e. $\mathbb{P}_{\mathbb{P}^1}(\mathscr{O}_{\mathbb{P}^1}\oplus \mathscr{O}_{\mathbb{P}^1}(e))$ with $e\geq 0$.\\
& &\\
\end{tabular}

Let $X$ be either a curve or a surface and let $D$ be a divisor on $X$. We will say that the divisor $D$ is \textit{special} on $X$ if the speciality is positive. Otherwise, we will say that $D$ is \textit{non-special}. A component $D'$ of a divisor $D$ on a surface $X$ and will be denoted $D'\subset D$.

The minimal models of rational surfaces are $\mathbb{P}^2$ and $\mathbb{F}_e$, where $e\neq 1$, and if $\pi$ is a birational morphism from $S$ to a minimal model then $\pi$ is a finite composition of blow ups centered at $x_1,\hdots, x_r\in\mathbb{P}^2$, and
$$\Pic(S)\simeq \begin{cases} \mathbb{Z}L\oplus ZE_1\oplus \hdots \oplus\mathbb{Z}E_r, &  \text{if }\mathbb{P}^2\text{ is the minimal model}\\
\mathbb{Z}B\oplus \mathbb{Z}F\oplus ZE_1\oplus \hdots \oplus\mathbb{Z}E_r, &  \text{if }\mathbb{F}_e\text{ is the minimal model}.
\end{cases}$$
\noindent where $E_i=\pi^{-1}(x_i)$ is an exceptional divisor on $S$, $L=\pi^*l$ for some line $l\subset \mathbb{P}^2$, $B$ is the class of a section with $B^2=e$ and $F$ is a fiber on the ruling.

The following form of the adjunction formula and Riemann-Roch will be useful.

\begin{lem}\label{binomLem} Let $S \simeq \tilde{\mathbb{P}}^2(x_1,\hdots, x_r)$ and let $\mathscr{O}_S(D)\in \Pic(S)$. The following are true:

\hskip55pt (1). Suppose $\mathbb{P}^2$ is a minimal model for $S$ and let $D\equiv aL-\sum b_iE_i$. Then:
\begin{eqnarray*}
p_a(D) & =& \tbinom{a-1}{2}-\sum \tbinom{b_i}{2}\\
\chi(\mathscr{O}_S(D)) & =& \tbinom{a+2}{2}-\sum \tbinom{b_i+1}{2}.
\end{eqnarray*}
\hskip65pt (2). Suppose $\mathbb{F}_e$ is a minimal model for $S$ and let $D\equiv aB+bF-\sum c_iE_i$. Then:
\begin{eqnarray*}
p_a(D) & =& (a-1)(b-1)+e\tbinom{a}{2}-\sum \tbinom{c_i}{2}\\
\chi(\mathscr{O}_S(D)) & =& (a+1)(b+1)+e\tbinom{a+1}{2}-\sum \tbinom{c_i+1}{2}.
\end{eqnarray*}
\end{lem}
\begin{proof} 
Recall that in (1) we have $K_S \equiv -3L+\sum E_i$ and in (2) we have $K_S\equiv -2B+(e-2)F+\sum E_i$.
The results about $p_a(D)$ follows from taking degree of the adjunction formula $\omega_D\simeq \mathscr{O}_D(D+K_S)$, Theorem 1.6.3 in \cite{5}. The results about $\chi(\mathscr{O}_S(D))$ follows directly from Riemann-Roch, Theorem V.1.6 in \cite{11}, combined with $\chi(\mathscr{O}_S)=1$. \end{proof}

By the \textit{type} of a divisor class we mean the short hand notations 
$$[a;\text{max}\{b_i\}^{u_0},\hdots ,\text{min}\{b_i\}^{u_v}]:=aL-b_1E_1-\hdots -b_rE_r$$
$$[(a,b);\text{max}\{c_i\}^{u_0},\hdots ,\text{min}\{c_i\}^{u_v}]:=aB+bF-c_1E_1-\hdots -c_rE_r$$ respectively, where $\sum_0^v u_i=r$ and $u_j=\#\{k\:|\: b_k=\text{max}\{b_j\}-j\}$. For instance, if $S \simeq \tilde{\mathbb{P}}^2(x_1,\hdots, x_r)$  and $\mathbb{P}^2$ is the minimal model for $S$, then the type of the anticanonical divisor class $|-K_S|$ is $[3;1^r]$.\\

\subsection{Preliminaries about very ample line bundles on smooth surfaces}
To construct a surface $S$ we seek line bundles $\mathscr{L}$ that are very ample on $S$. A versatile result, due to Alexander and Bauer, provides us with sufficient conditions for $\mathscr{L}$ to be very ample. We state their precise result.

\begin{lem}[Alexander-Bauer]\label{alexbau} An effective line bundle $\mathscr{O}_S(H)\simeq\mathscr{O}_S(D_1+D_2)$ on a smooth surface $S$ is very ample, if each one of the following is true:\\
\begin{tabular}{l l l}
(1). & $h^0(\mathscr{O}_S(D_i))\geq 2$, & for some $i$.\\
(2). & $\mathscr{O}_{D}(H)$ is very ample, & for all $D\in |D_1|\cup |D_2|$.\\
(3). & $H^0(\mathscr{O}_S(H))\longrightarrow H^0(\mathscr{O}_{D}(H))$ is surjective, & for all $D\in |D_1|\cup |D_2|$.
\end{tabular}
\end{lem}
\begin{proof}See Proposition 5.1 in \cite{6}. \end{proof}

The idea behind Alexander and Bauer's result is that if $\mathscr{L}$ restricts to a very ample line bundle on a suitable family of
curves on $S$ then $\mathscr{L}$ is itself very ample on $S$, given some minor assumptions. This allows us to answer
the question of $\mathscr{L}$ being very ample on $S$ by answering the question of $\mathscr{L}$ being very ample on some
curves on $S$. Note that the Alexander-Bauer lemma does not give any information on how 
to determine if $\mathscr{L}$ restricts to a very ample line bundle on curves. 

Recall that
it follows from Riemann-Roch that $\mathscr{L}$ is a very ample line bundle on an irreducible curve
$C$ if $\deg (\mathscr{L}\otimes \mathscr{O}_C)\geq 2p_a(C)+1$. Catanese, Franciosi, Hulek and Reid have generalized
the latter into a result which is true for both irreducible and reducible curves. The part of their result which we
will be using is the following. Note that all curves on a smooth surface $S$ are generically Gorenstein.

\begin{thm}[Curve embedding]\label{curvembd} Let $H$ be a divisor on a curve $C$ (possibly reducible and nonreduced). Then $\mathscr{O}_C(H)$ is very ample whenever $H.D\geq 2p_a(D)+1$, for every generically Gorenstein component $D\subset C$.
\end{thm}
\begin{proof}See Theorem 1.1 in \cite{7}. \end{proof}

Let $H$ be a very ample divisor on a $S$. Recall that the degree of a curve $C$ on $S$ is
$H.C>0$. Catanese and Franciosi have improved the lower bound of $H.D$ when $D$ is an effective divisor of small arithmetic genus. We state
their result.

\begin{prp}\label{lowgenPrp} 
Suppose $H$ is a very ample divisor on a smooth surface $S$. Then every effective divisor $D$ on $S$ with arithmetic genus $p_a(D)\leq 2$ has degree $H.D\geq 2p_a(D)+1$. In particular, if the degree $H.D\leq 3$ then the arithmetic genus $p_a(D)\leq 1$.
\end{prp}
\begin{proof}See Proposition 5.2 in \cite{6}. \end{proof}

Due to the usefulness of the decomposition $\mathscr{O}_S(H)\simeq\mathscr{O}_S(D_1+D_2)$ in the Alexander-Bauer lemma, we make the following definition.

\begin{definition}\label{niceDec} Let $H$ be an effective divisor on a smooth surface $S$. We say that
$H\equiv D_1+D_2$ is a \textit{nice decomposition} if $D_1$ and $D_2$ are both effective divisors on $S$ such that
$D_2$ is non-special on $S$, i.e. $h^1(\mathscr{O}_S(D_2))=0$, and the intersection product $H.D_1=2p_a(D_1)-2$. \end{definition}

\vskip5pt
\subsection{Preliminaries about dimensions of the cohomology groups of line bundles on rational surfaces}

\vskip5pt

\begin{lem}\label{forsLem} Let $D$ and $H$ be effective divisors on a smooth surface $S$. Then:\\
(1). $h^2(\mathscr{O}_S(D))=0$.\\
(2). Suppose $H.D>2p_a(D)-2$ and $h^1(\mathscr{O}_S(H-D))=0$, then $h^1(\mathscr{O}_D(H))=0$.\\
(3). Suppose $D^2>2p_a(D)-2$, then $h^1(\mathscr{O}_S(D))=0$.
\end{lem}

\begin{proof}
(1). Combine Serre duality, $H^2(\mathscr{O}_S(D))\simeq H^0(\mathscr{O}_S(K_S-D))$, with that $K_S$ is not effective.\\
(2). By assumption $h^1(\mathscr{O}_S(H-D))=0$ and by (1), the short exact sequence
$$0\longrightarrow \mathscr{O}_S(H-D)\longrightarrow\mathscr{O}_S(H)\longrightarrow \mathscr{O}_D(H)\longrightarrow 0$$
gives $h^1(\mathscr{O}_S(H))=h^1(\mathscr{O}_D(H))$. Then combine $H.D>2p_a(D)-2$ with Riemann-Roch.\\
(3). Apply (2) with $H\equiv D$ and recall that $h^i(\mathscr{O}_D)=0$, for all $i>0$. 
\end{proof}

We make the following observation about nice decompositions.

\begin{lem}\label{canLem}
Let $H\equiv D_1+D_2$ be a nice decomposition on a smooth surface $S$, where $H.D_1=2p_a(D_1)-2$ and $D_2$ is non-special on $S$. Then $h^1(\mathscr{O}_S(H))=1$ if and only if $\mathscr{O}_{D_1}(D_2-K_S)\simeq\mathscr{O}_{D_1}$.
\end{lem}
\begin{proof}
Suppose $h^1(\mathscr{O}_S(H))=1$. Taking cohomology of the short exact sequence
$$0\longrightarrow\mathscr{O}_S(D_2)\longrightarrow\mathscr{O}_S(H)\longrightarrow\mathscr{O}_{D_1}(H)\longrightarrow 0$$
gives $h^1(\mathscr{O}_S(H))=h^1(\mathscr{O}_{D_1}(H))=1$ since $D_2$ is non-special on $S$ and by Lemma \ref{forsLem}.1.
By Serre duality, $h^1(\mathscr{O}_{D_1}(H))=h^0(\mathscr{O}_{D_1}(K_{D_1}-H))=1$. Then the adjunction formula
yields $\mathscr{O}_{D_1}(K_{D_1})\simeq \mathscr{O}_{D_1}(D_1+K_S)$, such that $h^0(\mathscr{O}_{D_1}(K_S-D_2))=1$. Then $D_1.(K_S-D_2)=0$ implies that $ \mathscr{O}_{D_1}(K_S-D_2)\simeq \mathscr{O}_{D_2}$.

Conversely, suppose $\mathscr{O}_{D_1}(D_2-K_S)\simeq\mathscr{O}_{D_1}$. Twist the latter with $\mathscr{O}_{D_1}(D_1+K_S)$ and combine with the adjunction formula to obtain $\mathscr{O}_{D_1}(H)\simeq \omega_{D_1}$.  Then consider the short exact sequence $$0\longrightarrow \mathscr{O}_S(D_2)\longrightarrow\mathscr{O}_S(H)\longrightarrow\omega_{D_1}\longrightarrow 0$$
Since $D_2$ is non-special on $S$ and by Lemma \ref{forsLem}.1, we get $h^1(\mathscr{O}_S(H))=h^1(\omega_{D_1})=1$.
\end{proof}

\noindent 

\section{A construction}\label{constt}

We are now ready to construct a smooth rational surface with the prescribed numerical invariants. Furthermore,
this construction will prove the statement about existence in Theorem \ref{mainT}.

\begin{thm}\label{constEx1}
It is possible to choose points $x_1,\hdots ,x_5,y_1,y_2,z_1,\hdots ,z_{10}\in \mathbb{P}^2$ such that
the divisor class
$$H\equiv 7L-\sum_{i_1=1}^5 2E_i-\sum_{i_2=1}^2 2F_i-\sum_{i_3=1}^{10}G_i$$
is very ample on $S$ and $|H|$ embeds $S$ as a rational surface of degree $11$ and sectional genus $8$ in $\mathbb{P}^5$, such that
$$\pi:S\longrightarrow\mathbb{P}^2$$ denotes the morphism obtained by blowing up the points $x_1,\hdots ,x_5,y_1,y_2,z_1,\hdots ,z_{10}$,
$E_i=\pi^{-1}(x_i)$, $F_i=\pi^{-1}(y_i)$, $G_i=\pi^{-1}(y_i)$ and $L=\pi^*l$ where
$l\subset\mathbb{P}^2$ is a line.
\end{thm}

\begin{proof} We begin by choosing $x_1,\hdots, x_5\in\mathbb{P}^2$ in general position, in which case the open conditions\\
\\
\begin{tabular}{l l}
(O1). No three points $x_i$ are collinear.\\
\end{tabular}\\

\noindent are satisfied. Let $\pi_1:S_1\longrightarrow\mathbb{P}^2$
denote the morphism obtained by blowing up $x_1,\hdots ,x_5$ and denote $E_i=\pi_1^{-1}(x_i)$. On the rational
surface $S_1$ we study the complete linear systems associated to the following two divisor classes
\begin{eqnarray*}
-2K_{S_1} & \equiv & 6L-2E_1-\hdots  -2E_5\\
L-K_{S_1} & \equiv & 4L-E_1-\hdots -E_5.
\end{eqnarray*}

\noindent Since $(S_1,-K_{S_1})$ is a quartic Del Pezzo surface, the anti-canonical divisor $-K_{S_1}$
is very ample on $S_1$. In particular, this means that $-2K_{S_1}$ is very ample.
Furthermore, since $|L|$ is base-point free it follows that 
$L-K_{S_1}$ is very ample  by using the Segre embedding
and recalling that $\mathscr{O}_{\mathbb{P}^{N_1}\times\mathbb{P}^{N_2}}(1)\simeq p_1^*\mathscr{O}_{\mathbb{P}^{N_1}}(1)\otimes 
p_2^*\mathscr{O}_{\mathbb{P}^{N_2}}(1)$, where $p_i:\mathbb{P}^{N_1}\times\mathbb{P}^{N_2}\longrightarrow\mathbb{P}^{N_i}$
is the $i$-th projection map. So 
$-2K_{S_1}$ and $L-K_{S_1}$ are very ample on $S_1$. By Bertini's theorem, Theorem 20.2 in \cite{5}, a general
choice of curves in $|-2K_{S_1}|$ and $|L-K_{S_1}|$ are both irreducible and smooth.
\end{proof}

Next we choose the points $y_1,y_2,z_1,\hdots ,z_{10}\in\mathbb{P}^2$ such that the complete linear systems
\begin{eqnarray*}
\Delta_1 & = & |6l-\sum 2x_i-\sum y_i|
\\
\Delta_2 & = & |4l-\sum x_i-\sum y_i|
\end{eqnarray*}
on $\mathbb{P}^2$ satisfy the following open condition\\
\\
\begin{tabular}{l l}
(O2). & The pair of points $y_1$ and $y_2$ in $\mathbb{P}^2$, the pair of tangent directions $y_1'$ and $y_2'$ at these points\\
& and a pair of curves $(D_1,D_2)\in \Delta_1\times \Delta_2$ satisfy the condition that the intersection \\
& $D_1\cap D_2=\sum x_i+\sum y_i+\sum y_i'+\sum z_i$, where the $z_1,...,z_{10}$ are distinct closed points \\
& disjoint from the $x_i$ and the $y_i$.
\end{tabular}\\

Note that O2 is an open condition on $y_i, y_i'$ and $D_1$,$D_2$, while it implies that the points $z_1,...,z_{10}$ are in a special position. Furthermore, we claim that O2 is a non-empty condition. That is, we claim that it's possible to choose the points $y_1,y_2,z_1,\hdots,z_{10}\in\mathbb{P}^2$ on the intersection of a smooth quartic and a smooth sextic sharing tangent directions in the points $y_1,y_2$.\\
\\
\textit{Claim: The open condition O2 is non-empty.}\\
\\
Start off by choosing a smooth and irreducible curve $A_1\in |-2K_{S_1}|$, recall that this is possible due to Bertini's theorem, and consider the incidence
$\Lambda\subset A_1\times A_1\times |L-K_{S_1}|$, given by
$$\Lambda =\left\{(y_1,y_2,B_1)\:|\: y_1,y_2\in A_1\cap B_1,\:A_1\text{ and }B_1\text{ share common tangent directions at }y_1\text{ and }y_2\right\}.$$
Now, choose a triple
$(y_1,y_2,B_1)\in \Lambda$ such that $B_1$ is smooth and irreducible on $S_1$. Recall that this is possible due to Bertini's theorem. Since $|B_{ 1|A_1}|$ is base-point free, it follows that the curve $B_1$ is smooth
at each closed point in the zero-dimensional scheme $A_1\cap B_1$, and that the intersection is transversal except at $y_1$ and $y_2$.
So we may set
$$A_1\cap B_1=\sum_1^5 x_i+\sum_1^2 y_i+\sum_1^2 y_i'+\sum z_i$$
where $\{z_i\}$ are remaining points on the intersection of $A_1$ and $B_1$. Furthermore, there are 
$\#\{z_i\}=6\cdot 4-2\cdot 5-4=10$ number of distinct points in $\{z_i\}$. 

Now we blow up the points $y_1,y_2,z_1,\hdots z_{10}\in S_1$
to obtain a morphism $\pi_2:S\longrightarrow S_1$, such that we may define $\pi=\pi_1\circ\pi_2:S\longrightarrow\mathbb{P}^2$. Denote
\begin{eqnarray*}
A & \equiv & 6L-2E_1-\hdots -2E_5-F_1-F_2-G_1-\hdots -G_{10} \\
B & \equiv & 4L-E_1-\hdots -E_5-F_1-F_2-G_1-\hdots -G_{10} 
\end{eqnarray*}
as the strict transforms of $A_1$ and $B_1$ on $S$, respectively. Note that by our construction the divisors $A$ and $B$ are smooth and irreducible on the surface $S$, due to Bertini's theorem.

We are now ready to state the remaining sufficient conditions for  Theorem \ref{constEx1}
on the configuration of the points $x_i,y_j, z_k$. 
In addition to the open conditions O1-O2 we assume that the following open conditions on $S$ are satisfied.\\
\\
\begin{tabular}{l l l}
(O3). & $|L-\sum_{i\in I}E_i-F_j|=\emptyset$, & for $|I|\geq 2$ and for all $j$.\\
(O4). & $|2L-E_1-\hdots -E_5-F_j|=\emptyset$, & for all $j$.\\
(O5). & $|6L-2E_1-\hdots 2E_5-2F_{1}-2F_{2}-G_1-\hdots -G_{10}|=\emptyset$. & \\
(O6). & $|6L-2E_1-\hdots 2E_5-3F_{j}-F_{3-j}-G_1-\hdots -G_{10}|=\emptyset$, & for all $j$.\\
\end{tabular}
\vskip 3pt

It is straightforward but tediuos to check that these open conditons are satisfied by some surface $S$.

\begin{lem} The open conditions O1, O3-O6 are necessary conditions for $H$ to be very ample.
\end{lem}
\begin{proof}
Note that if at least one of the linear systems in O1,O3-O4 are non-empty, then there exists an effective $D$
on the surface $S$ such that $H.D\leq 0$.  All divisors $D$ in the linear systems of O5 or O6 satisfy $H.D=2$ and $p_a(D)\geq 2$. Therefore, O5 and O6 are also necessary.
\end{proof}

On the surface $S$, we study the divisor class of $A$ and the following two divisor classes:
\begin{eqnarray*}
C & \equiv & L-F_1-F_2 \\
H:=A+C & \equiv & 7L-2E_1-\hdots -2E_5-2F_1-2F_2-G_1-\hdots -G_{10}. 
\end{eqnarray*}

The rest of the argument is based on proving two things. First, we claim that $H$ admits exactly six global sections on $S$. Second, we claim that $H$ is very ample on $S$.

\begin{lem}\label{Herbra}
The decomposition $H\equiv A+C$ is nice.
\end{lem}
\begin{proof}
Note that $C$ corresponds to a line between two points, so $C$ is unique and non-special on $S$.
Also, $\chi (\mathscr{O}_S(A))=1$ and $\chi(\mathscr{O}_S(A))\leq h^0(\mathscr{O}_S(A))$ yields that $A$ is indeed effective on $S$. Furthermore, $H.A=2p_a(A)-2$. Thus $H\equiv A+C$ is indeed a nice decomposition. 
\end{proof}

\noindent We are ready to show that $\mathscr{O}_S(H)$ admits six global sections on the surface $S$.

\begin{lem}\label{const1Can} $\mathscr{O}_A(C-K_S)\simeq\mathscr{O}_A$ and $h^0(\mathscr{O}_S(H))=6$.
\end{lem}
\begin{proof}
By our construction, $A\cap B=y_1'+y_2'$ on the surface $S$.
Since tangent directions $y_1',y_2'$ at the points $y_1,y_2\in S$ corresponds to points on the exceptional divisors $F_1,F_2\subset S$. we have $\mathscr{O}_A(B)\simeq \mathscr{O}_A(F_1+F_2)$. That is, $\mathscr{O}_A(B-F_1-F_2)\simeq\mathscr{O}_A$ which combined with $\mathscr{O}_A(C-K_S)\simeq \mathscr{O}_A(B-F_1-F_2)$ gives the first statement. Finally, combining Lemma \ref{canLem} with $\chi(\mathscr{O}_S(H))=5$ we obtain $h^0(\mathscr{O}_S(A))=6$. \end{proof}

\noindent To finish the proof it remains to show that $H$ is very ample on $S$. We show this by applying
the Alexander-Bauer lemma and the curve embedding theorem to the decomposition $H\equiv A+C$. This
requires us to determine the dimension of the complete linear system $|A|$.

\begin{lem}\label{dimA1}
$h^0(\mathscr{O}_S(A))=2$.
\end{lem}
\begin{proof}
Taking the union of the $B$ and the unique quadric $Q$ passing through $x_1,\hdots, x_5$, we get that $B+Q\in |A|$. On the other hand, $A$ is chosen to be irreducible by construction so we have $h^0(\mathscr{O}_S(A))\geq 2$. In particular, this means that $h^1(\mathscr{O}_S(A))\geq 1$ since $\chi(\mathscr{O}_A(S))=1$. Now, consider the short exact sequence
$$0\longrightarrow\mathscr{O}_S\longrightarrow\mathscr{O}_S(A)\longrightarrow\mathscr{O}_A(A)\longrightarrow 0.$$
The rationality of $S$ implies that $h^1(\mathscr{O}_S(A))=h^1(\mathscr{O}_A(A))$ so that $\mathscr{O}_A(A)$ is special. By Clifford's theorem, Theorem IV.5.4 in \cite{11}, and the short exact sequence above we obtain
$$\displaystyle h^0(\mathscr{O}_S(A))=h^0(\mathscr{O}_A(A))-h^0(\mathscr{O}_S)\leq \frac{1}{2}A^2+1-1=2.$$
Therefore, $h^0(\mathscr{O}_S(A))=2$.\end{proof}

\noindent 
Next we show that every assumption in the Alexander-Bauer lemma, except the very ampleness of $H_{|D}$, for all $D\in |A|$,
is satisfied. Then we proceed to show that $H_{|D}$ is very ample, for all $D\in |A|$, by using the Curve embedding theorem.

\begin{lem} The complete linear system $|H|$ restricts to a very ample linear system on $C$ and
the restriction maps $H^0(\mathscr{O}_S(H))\longrightarrow H^0(\mathscr{O}_C(H))$
and $H^0(\mathscr{O}_S(H))\longrightarrow H^0(\mathscr{O}_D(H))$ are surjective, for all $D\in |A|$.
\end{lem}
\begin{proof}
The first assertion is true since $C$ is a smooth rational curve and $H.C>0$ . For the second assertion,
consider the short exact sequences
$$0\longrightarrow\mathscr{O}_S(C)\longrightarrow\mathscr{O}_S(H)\longrightarrow\mathscr{O}_D(H)\longrightarrow 0$$
$$0\longrightarrow\mathscr{O}_S(A)\longrightarrow\mathscr{O}_S(H)\longrightarrow\mathscr{O}_C(H)\longrightarrow 0$$
Since $C$ is a non-special curve on $S$, it follows from the first short exact sequence that 
$H^0(\mathscr{O}_S(H))\longrightarrow H^0(\mathscr{O}_D(H))$ is surjective, for all $D\in |A|$.
For the last assertion, the long exact sequence associated to the second short exact sequence is
$$\hdots  \rightarrow H^0(\mathscr{O}_S(H))\xrightarrow{\alpha} H^0(\mathscr{O}_C(H)) \xrightarrow{\beta} H^1(\mathscr{O}_S(A))\xrightarrow{\gamma} H^1(\mathscr{O}_S(H))\rightarrow 0$$
since $h^1(\mathscr{O}_C(H))=0$. Now, since $\gamma$ is surjective and $h^1(\mathscr{O}_S(A))=h^1(\mathscr{O}_S(H))$ due to Lemma \ref{const1Can} and \ref{dimA1}, $\ker(\gamma)=\im(\beta)=0$. Thus $\alpha$ is surjective.
\end{proof}

\noindent To show that $\mathscr{O}_D(H)$ is very ample, for all $D\in |A|$, we partition $|A|$ into
the following two families of curves on $S$ and consider each family separately. 

$$\mathscr{A}_{\text{Good}}=\{ D\in |A|\: :\:\text{Every subcurve }A'\leq D\text{ satisfies }A'.F_j\leq 1,\:\text{for all }j\}.$$
$$\mathscr{A}_{\text{Bad}}=\{ D\in |A|\: :\:\text{Some subcurve }A'\leq D\text{ satisfies }A'.F_j>1,\:\text{for some }j\}.$$

\noindent The idea behind the partition is that the curves $D$ meeting the exceptional divisors $F_j$ at most once, i.e. $D\in\mathscr{A}_{\text{Good}}$, are isomorphic to their image under the blow-up $\pi_2:S\longrightarrow S_1$, making it possible to utilize the morphism $\pi_2$ to show that $\mathscr{O}_D(H)$ are very ample.\\
\\
\noindent First we consider the curves in $\mathscr{A}_{\text{Bad}}$. Denote $$D_j  \equiv  A-F_j,$$ where $1\leq j\leq 2$. We claim that $\mathscr{A}_{\text{Bad}}$ consists of exactly the two divisors $D_1$ and $D_2$.

\begin{lem}\label{badChar}
Let $D\in \mathscr{A}_{\text{Bad}}$. Then $D$ is reducible and $D=D'+F_j$, for some $D'\in |D_j|$ and for some $j$. Furthermore, $D'$ does not contain $F_j$ as a component.
\end{lem}
\begin{proof} 
It is clear that $D$ is reducible since there exists a component $A'$ of $D$ such that $A'.F_j>1$, for some $j$. Then some $F_j$ is a component of $D$ and that component cannot have multiplicity $>1$ due to the open condition O5.
\end{proof}

\noindent Next, we show that the pencil $|A|$ has base points on $F_1$ and $F_2$ by showing that $|D_1|$ and $|D_2|$ are indeed non-empty on $S$.

\begin{lem}\label{badEff}
$h^0(\mathscr{O}_S(D_j))=1$, for all $1\leq j\leq 2$.
\end{lem}
\begin{proof}
Let $Q$ be the unique conic passing through $x_1,\hdots, x_5$. Without loss of generality, we may assume that $j=1$. Consider the short exact sequence
$$0\longrightarrow \mathscr{O}_S(Q-F_1)\longrightarrow \mathscr{O}_S(D_1)\longrightarrow\mathscr{O}_B(D_1)\longrightarrow 0$$
From the open condition O5 and $\chi(\mathscr{O}_S(Q-F_1))=0$, it follows that $h^0(\mathscr{O}_S(D_1))=h^0(\mathscr{O}_B(D_1))$. Recall that by construction, $\mathscr{O}_B(A)\simeq \mathscr{O}_B(F_1+F_2)$. The latter implies that $\mathscr{O}_B(D_1)\simeq\mathscr{O}_B(F_{2}).$
Since $B.F_2=1$ and $B$ is not a rational curve, we have $h^0(\mathscr{O}_B(F_2))=1$. That is, $h^0(\mathscr{O}_S(D_1))=1$.
\end{proof}

\begin{lem}\label{eksEmb}
$\mathscr{O}_D(H)$ is very ample, for all $D\in \mathscr{A}_{\text{Bad}}$.
\end{lem}
\begin{proof}
Due to Lemma \ref{badChar} and Lemma \ref{badEff} it suffices to show that $\mathscr{O}_{D}(H)$ is very ample, 
when $D=D_1\cup F_1$ and $D=D_2\cup F_2$.

 Note that
$\mathscr{O}_{F_j}(H)$ is very ample since $H.F_j>2p_a(F_j)+1$.
 Next, we claim that
the $\mathscr{O}_{D_j}(H)\simeq\omega_{D_j}$. Taking cohomology of the short exact
sequence
$$0\longrightarrow\mathscr{O}_S(C+F_j)\longrightarrow\mathscr{O}_S(H)
\longrightarrow\mathscr{O}_{D_j}(H)\longrightarrow 0$$
we note that $C+F_j$ is a non-special curve on $S$, such that
$h^1(\mathscr{O}_{D_j}(H))=h^1(\mathscr{O}_S(H))=1$ due to Lemma \ref{const1Can}. Combining
the latter with $H.D_j=2p_a(D_j)-2$ we conclude that $\mathscr{O}_{D_j}(H)
\simeq\omega_{D_j}$. Then the adjunction formula yields that
$$\mathscr{O}_{D_j}(H)\simeq\mathscr{O}_{D_j}(D_j+K_S)\simeq
\mathscr{O}_{D_j}(3L-E_1-\hdots -E_5-F_j).$$
By conditions O1,O3,O4, the linear system $|3L-E_1-\hdots -E_5-F_j|$ contracts only exceptional curves on $S$ that intersect $D_j$ in at most one point, so it, and hence $|H|$ is very ample on $D_j$.  It embeds $D_j$ as a complete intersection of a quadric and a cubic surface in a $\mathbb{P}^3$.

Finally,
we show that $|H|$ is very ample on the union $D_j\cup F_j$. The intersection $D_j\cap F_j$ is a pair of points that span a line.  Since $\varphi_H(F_j)$ is a conic, and $\varphi_H(D_j)$ is a space curve, the linear system  
$|H|$ is very ample on the union $D_j\cup F_j$ if the union $\varphi_H(F_j)\cup \varphi_H(D_j)$ spans a $\mathbb{P}^4$.
For this, note that $D_j+F_j\equiv A$, so 
$$\dim (\varphi_H(D_j)\cup \varphi_H(F_j)) = h^0(\mathscr{O}_{A}(H))-1=4.$$
Thus $\mathscr{O}_A(H)$ is very ample, for all $\mathscr{A}_{\text{Bad}}$.\end{proof}

\noindent Finally, we consider the curves in $\mathscr{A}_{\text{Good}}$.

\begin{lem}
$\mathscr{O}_D(H)$ is very ample, for all $D\in \mathscr{A}_{\text{Good}}$.
\end{lem}
\begin{proof}
Let $D'$ be a component of some $D\in \mathscr{A}_{\text{Good}}$.
It is clear that the blow-up morphism $\pi_2:S\longrightarrow S_1$ defines an isomorphism
$D'\simeq \pi_2(D')$. On the other hand,
since 
$$(\pi_2(D)+K_{S_1})\cdot F_i=(\pi_2(D)+K_{S_1})\cdot G_j=0$$
for all $1\leq i\leq 2$ and $1\leq j\leq 10$, it follows that
the $D+K_S\equiv \pi_2^*(\pi_2(D)+K_{S_1})$. 
But $\pi_2(D)\equiv -2K_{S_1}$, so that
$$\mathscr{O}_D(D+K_S)\simeq\pi_2^*\mathscr{O}_{\pi_2(D)}(\pi_2(D)+K_{S_1})\simeq\pi_2^*\mathscr{O}_{\pi_2(D)}(-K_{S_1}).$$
By Lemma \ref{const1Can}, we have
$$\mathscr{O}_D(H)\simeq\mathscr{O}_D(H-C+K_S)\simeq\mathscr{O}_D(D+K_S).$$

Therefore, $\mathscr{O}_D(H)$ is very ample whenever 
$\mathscr{O}_{D"}(-K_{S_1})$ is very ample for every curve in $D"\in |-2K_{S_1}|$. 
Now, recall that $\mathscr{O}_{S_1}(-K_{S_1})$ is indeed very ample, on $S_1$ and hence on any curve $D"$, since $(S_1,-K_{S_1})$
is a quartic Del Pezzo surface.

\noindent This concludes the proof of Theorem \ref{constEx1}. \end{proof}

\section{Lifting examples}\label{liftSect}

We present two examples of linear systems 
$|H|$ that cannot be both special and very ample on a surface with the prescribed numerical invariants. These examples will be used to rule out possible adjunction maps in the classification done in Theorem \ref{kuttThm}.

The idea behind the examples is to search for nice decompositions, $\mathscr{O}_S(H)\simeq \mathscr{O}_S( A_1+A_2)$, such that
some divisor $C$ on $S$ restricts to the trivial bundle on one of the components $A_i$, and the section of 
$\mathscr{O}_{A_i}(C)$ lifts to a section of $\mathscr{O}_S(C)$. On the other hand, we will establish that the effectiveness of $\mathscr{O}_S(C)$ contradicts the very ampleness of $\mathscr{O}_S(H)$.

\begin{prp}\label{liftEx1}
Let $S$ be a smooth rational surface and let $\pi:S\longrightarrow\mathbb{P}^2$
denote the morphism obtained by blowing up the points $x_1,\hdots ,x_{15}\in \mathbb{P}^2$. Furthermore,
let $E_i=\pi^{-1}(x_i)$ and $L=\pi^*l$, where $l\subset\mathbb{P}^2$ is a line. Suppose the divisor class
$$H\equiv 10L-4E_1-\sum_{i=2}^8 3E_i-2E_9-\sum_{j=10}^{15}E_j$$
has $h^0(\mathscr{O}_S(H))=6$. Then $H$ is not very ample on $S$.
\end{prp}
\begin{proof}
We consider the following complete linear system on the surface $S$.
$$\Delta :=|6L-2E_1-\hdots -2E_8-E_9-\hdots -E_{15}|$$
The idea is to show that $h^0(\mathscr{O}_S(H))=6$ implies $\Delta\neq \emptyset$. To do this, we study a nice decomposition of $H$ from which we construct an effective divisor on $\Delta$. In particular, we consider $H\equiv A+B$ where
\begin{eqnarray*}
A & \equiv & 7L-3E_1-2E_2-\hdots -2E_9-E_{10}-\hdots -E_{15}\\
B & \equiv & 3L-E_1-\hdots -E_8. 
\end{eqnarray*}

\begin{lem} Suppose $h^0(\mathscr{O}_S(H))=6$. Then $H\equiv A+B$ is a nice decomposition.
\end{lem}
\begin{proof}
\noindent Clearly $h^0(\mathscr{O}_S(B))>0$, but
there is a priori no reason to believe that $|A|$ is a non-empty complete linear system on $S$ since $\chi(\mathscr{O}_S(A))=0$. However, consider the short exact sequence
$$0\longrightarrow\mathscr{O}_S(A)\longrightarrow\mathscr{O}_S(H)\longrightarrow\mathscr{O}_B(H)\longrightarrow 0.$$
Since $H.B>2p_a(B)-2$, it follows that $h^1(\mathscr{O}_B(H))=0$. Then the surjectivity of $H^1(\mathscr{O}_S(A))\longrightarrow H^1(\mathscr{O}_S(H))$ yields that
$h^1(\mathscr{O}_S(A))\geq h^1(\mathscr{O}_S(H))$.
Therefore $h^0(\mathscr{O}_S(A))=h^1(\mathscr{O}_S(A))\geq 1$, i.e. $|A|$ is a non-empty complete linear system on $S$.
Furthermore, note that $H.A=2p_a(A)-2$ and that $h^1(\mathscr{O}_S(B))=0$. So, the decomposition 
$H\equiv A+B$ is indeed a nice decomposition.\end{proof}

\noindent We proceed by showing that it is possible to lift the non-zero global section of $\mathscr{O}_A(B-K_S)$ to the surface $S$ and thus showing that $\Delta$ is a non-empty complete linear system on $S$. 
\begin{lem}\label{lift1} Suppose $h^0(\mathscr{O}_S(H))=6$. Then $\Delta$ is non-empty. 
\end{lem}
\begin{proof}
By Lemma \ref{canLem} we have $\mathscr{O}_A(B-K_S)\simeq\mathscr{O}_A$. Consider the following short exact sequence
$$0\longrightarrow\mathscr{O}_S(B-K_S-A)\longrightarrow\mathscr{O}_S(B-K_S)\longrightarrow\mathscr{O}_A\longrightarrow 0.$$
Note that $B-K_S-A\equiv -L+E_1+E_9$ is not effective on $S$. Furthermore, $h^0(\mathscr{O}_{L-E_1-E_9})=1$ since $L-E_1-E_9$ is connected, so that $h^1(\mathscr{O}_S(B-K_S-A))=0$. Then the short exact sequence implies that
$h^0(\mathscr{O}_S(B-K_S))=h^0(\mathscr{O}_A)=1$. \end{proof}

\begin{lem}\label{va1}
Suppose $\Delta$ is non-empty. Then $H$ is not very ample.
\end{lem}
\begin{proof}
Let $C\in\Delta $. Note that $p_a(C)=2$ and $H.C=2$. If $H$ is very ample, then $H$ embeds $C$ as a conic or the union of two lines. Both cases contradict $p_a(C)=2$.
\end{proof}

\noindent This proves Proposition \ref{liftEx1}.\end{proof}

\vskip10pt
\begin{prp}\label{liftEx2}
Let $S$ be a smooth rational surface and let $\pi:S\longrightarrow\mathbb{P}^2$
denote the morphism obtained by blowing up the points $x_1,x_2,y_1\hdots ,y_{17}\in \mathbb{P}^2$. Furthermore,
let $E_i=\pi^{-1}(x_i)$, $F_i=\pi^{-1}(y_i)$ and $L=\pi^*l$, where $l\subset\mathbb{P}^2$ is a line. Suppose the divisor class
$$H\equiv 6L-\sum_{i=1}^2 2E_i-\sum_{i=1}^{17}F_i$$
has $h^0(\mathscr{O}_S(H))=6$. Then $H$ is not very ample on $S$.
\end{prp}
\begin{proof}
We consider the following complete linear systems on the surface $S$.
\begin{center}
\begin{tabular}{l l}
$\Delta_1:=|L-\sum_{i\in I}E_i-\sum_{j\in J}F_j|$, & where $2|I|+|J|\geq 6$.\\
$\Delta_2:=|2L-\sum_{i\in I}E_i-\sum_{j\in J}F_j|$, & where $2|I|+|J|\geq 12$.\\
$\Delta_3:=|4L-E_1-E_2-F_1-\hdots -F_{17}|$. & \\
\end{tabular}
\end{center}

\noindent The idea now is to show that
if $h^0(\mathscr{O}_S(H))=6$, then at least one of the complete linear systems $\Delta_k$ is non-empty. To do so, we
study the decompositions $H\equiv A_{ij}+B_{ij}$, where $1\leq i\leq 2$, $1\leq j\leq 17$ and
\begin{eqnarray*}
A_{ij} & \equiv & 5L-2E_1-2E_2-F_1-\hdots -F_{17}+E_i+F_j\\
B_{ij} & \equiv & L-E_i-F_j.
\end{eqnarray*}
Note that the decompositions $H\equiv A_{ij}+B_{ij}$ are all nice decompositions due to Riemann-Roch, $B_{ij}$ are all non-special,
and that $H.A_{ij}=2p_a(A_{ij})-2$. 

\noindent 
\begin{lem}\label{lift2} Suppose $h^0(\mathscr{O}_S(H))=6$. Then the divisors $L-K_S+F_j$ are effective on $S$, for all $j$.
\end{lem}
\begin{proof}
 Without loss
of generality, we let $A_{j}:=A_{2j}$ and $B_j:= B_{2j}$. Recall that Lemma \ref{canLem} gives us $\mathscr{O}_{A_j}(B_j-K_S)\simeq\mathscr{O}_{A_j}$. Twisting the following short exact sequence
$$0\longrightarrow\mathscr{O}_S(B_j-K_S-A_j)\longrightarrow\mathscr{O}_S(B_j-K_S)\longrightarrow\mathscr{O}_{A_j}\longrightarrow 0$$
with $\mathscr{O}_S(E_2+2F_j)$, we get another short exact sequence.
$$0\longrightarrow\mathscr{O}_S(-L+E_1)\longrightarrow\mathscr{O}_S(L-K_S+F_j)\longrightarrow\mathscr{O}_{A_j}(E_2+2F_j)\longrightarrow 0.$$ 
Note that $-L+E_1$ is not effective on $S$. Furthermore, $h^0(\mathscr{O}_{L-E_1})=1$ since $L-E_1$ is connected, so that $h^1(\mathscr{O}_S(-L+E_1))=0$. In particular, this means that $h^0(\mathscr{O}_S(L-K_S+F_j))=h^0(\mathscr{O}_{A_j}(E_2+2F_j)).$
Next, combining the injection $H^0(\mathscr{O}_S(E_2+F_j))\hookrightarrow H^0(\mathscr{O}_S(E_2+2F_j))$ with $E_2+F_j$ being effective on $S$, it follows that $E_2+2F_j$ is effective on $S$. Then the short exact sequence 
$$0\longrightarrow \mathscr{O}_S(E_2+2F_j-A_j)\longrightarrow\mathscr{O}_S(E_2+2F_j)\longrightarrow\mathscr{O}_{A_j}(E_2+2F_j)\longrightarrow 0$$
yields that $h^0(\mathscr{O}_S(L-K_S+F_j))=h^0(\mathscr{O}_{A_j}(E_2+2F_j))>0$ since $h^0(\mathscr{O}_S(E_2+2F_j-A))=0$. \end{proof}

\noindent Denote $C_j:= L-K_S+F_j$. Then
$$C_j \equiv  4L-E_1-E_2-F_1-\hdots -F_{17}+F_j$$
are effective divisors on $S$, for all $j$, by Lemma \ref{lift2}. We are now ready to show how this implies that $H$ is not very ample
and thus proving the Proposition. 

\begin{lem}
Suppose $C_j$ are effective, for all $j$. Then $H$ is not very ample.
\end{lem}
\begin{proof}
The images of the curves $C_j$ under the blowdown morphism $\pi: S\longrightarrow\mathbb{P}^2$
are plane quartic curves passing through the points $x_1,x_2$ and $16$ of the points $y_1,\hdots ,y_{17}$. Bezout's theorem
implies that every two curves $\pi(C_i)$ and $\pi(C_j)$ share a common component, since the points $x_1,x_2$ and $15$ of the points
$y_1,\hdots ,y_{17}$ lie on the set-theoretic intersection $\pi(C_i)\cap \pi(C_j)$, whenever $i\neq j$. In particular, this means that there exists a plane
quartic passing through the points $x_1,x_2,y_1,\hdots ,y_{17}$ which implies that the divisor
$$
C \equiv 4L-E_1-E_2-F_1-\hdots -F_{17}
$$
is effective on the surface $S$. Now, suppose $H$ is very ample. Then the complete linear system $|H|$ embeds $C$ as a cubic curve due to $H.C=3$. The irreducible components of the curve $C$ are one of the following three cases: (1). Three lines. (2). A conic and a line. (3). Plane cubic curve.

In the first case, by pigeonholing the points $19$ points among the $3$ lines, clearly at least one of the lines
passes through at least $7$ of the exceptional divisors which in turn implies that $\Delta_1$ is non-empty. In the second case,
again by pigeonholing the points, it follows that either the conic passes through at least $12$ of the points or the line passes through
at least $8$ of the points, such that $\Delta_1$ or $\Delta_2$ is non-empty. In both cases, $H.D\leq 0$ for all $D\in |\Delta_1|\cup |\Delta_2|$ such that $H$ is not very ample. Finally, the third case occurs when $\Delta_3$ contains an irreducible plane cubic of genus three which is absurd.
\end{proof}
\noindent This proves Proposition \ref{liftEx2}.
\end{proof}

\section{Classification}\label{adjSect}

Let $S$ be a smooth rational surface and let
$i:S\hookrightarrow\mathbb{P}^n$ be a projective embedding with hyperplane section $H$. To carry out the classification we determine explicit expressions
for $H$, whenever the degree $H^2=11$ of the surface and the sectional genus $\pi_S=8$. To do so we study the adjoint
linear system $|H+K_S|$ on $S$ instead and make use of a result, due to Sommese and Van de Ven, which states
that $|H+K_S|$ almost always defines a birational morphism to some projective space. 

\begin{thm}\label{adj}
Let $S$ be a smooth rational surface in $\mathbb{P}^n$, let $H$ denote the
class of a hyperplane section of $S$, let $K_S$ denote the class of a canonical divisor
of $S$ and let $\tilde{\mathbb{P}}^2(x_1,..,x_r)$ denote $\mathbb{P}^2$ blown up in the points $x_1,\hdots , x_r$.  Then
the adjoint linear system $|H+K_S|$ defines a birational morphism
$$\varphi_{|H+K_S|}:S\longrightarrow\mathbb{P}^N$$
onto a smooth surface $S_1$ and $\varphi_{|H+K_S|}$
blows down $(-1)$-curves $E$ on $S$ such that \begin{equation*} K_S.E=-1\text{ and }H.E=1,\end{equation*}unless one of the following three cases occurs:\\
\begin{tabular}{l l}
(1). & $S$ is a plane, or $S$ is a Veronese surface of degree $4$, or $S$ is ruled by lines.\\
(2). & $(H+K_S)^2=0$, which occurs if and only if $S$ is a Del Pezzo surface or $\varphi$ is a conic bundle.\\
(3). & $(H+K_S)^2>0$ and $S$ belongs to one of the following four families:\\
& \begin{tabular}{l l}
(i). & $S=\tilde{\mathbb{P}}^2(x_1,..,x_7)$ is embedded by $H\equiv 6L-2E_1-\hdots -E_7$.\\
(ii). & $S=\tilde{\mathbb{P}}^2(x_1,..,x_8)$ is embedded by $H\equiv 6L-2E_1-\hdots -2E_7-E_8$.\\
(iii). & $S=\tilde{\mathbb{P}}^2(x_1,..,x_8)$ is embedded by $H\equiv 9L-3E_1-\hdots -3E_8$.
\end{tabular}
\end{tabular}
\end{thm}
\begin{proof}
See \cite{13} for a proof. 
\end{proof}

A crucial difference between surfaces in $\mathbb{P}^4$ and surfaces in $\mathbb{P}^5$ is that every
smooth surface can be embedded into $\mathbb{P}^5$ through generic projection. In particular, this implies that
we cannot expect a relation between invariants of a surface in $\mathbb{P}^5$ similar to the double-point formula of $\mathbb{P}^4$, 
see Example A4.1.3 in \cite{11}, which states that $c_{\codim(S,\mathbb{P}^4)}(\mathscr{N}_{S|\mathbb{P}^4})-(\deg S)^2=0$,
when $S\subset \mathbb{P}^4$. The double-point formula for surfaces in $\mathbb{P}^4$ plays an important role since
it completely determines $K_S^2$ whenever $\deg S$ and $\pi_S$ are given. Since we are considering surfaces in $\mathbb{P}^5$, our
strategy is to limit ourselves to finitely many possibilities for $K_S^2$, whenever we are given $\deg S$ and $\pi_S$. In the next result,
we give an upper and lower bound for $K_S^2$ and show some other generalities on the adjunction mapping.

\begin{lem}\label{adjLem} Suppose $H$ is very ample divisor on $S$ and suppose $\varphi_{|H+K_S|}$ is a birational morphism onto $S_1$. Then:\\
\begin{tabular}{l l}
(1). & $\varphi_{|H+K_S|}$ maps $S$ into $\mathbb{P}^{\pi_S-1}$.\\
(2). & $\pi_S-2-H.(H+2K_S)\leq  K_S^2\leq \lceil (H.K)^2/H^2\rceil$.\\
(3). & $(H+K_S).K_{S_1}=(H+K_S).K_S$\\
(4). & $\pi_{S_1}=\pi_S+(H+K_S).K_S$\\
(5). & If $H.K\geq -2$, then $K^2<0$.
\end{tabular}
\end{lem}
\begin{proof}
(1). Combine Riemann-Roch and the adjunction formula to get $\chi(S,\mathscr{O}_S(H+K_S))=\pi_S$.
The very ampleness of $H$ and smoothness of $S$ implies that $H^1(S,\mathscr{O}_S(H+K_S))=0$, due to
the Kodaira vanishing theorem. Furthermore, $H^2(S,\mathscr{O}_S(H+K_S))=0$ follows from the rationality
of $S$.\\
(2). The upper bound for $K_S^2$ is a direct consequence of the Hodge index inequality.
The lower bound for $K_S^2$ is obtained by noting that the non-degeneracy of $S$ yields that $\codim(S,\mathbb{P}^{\pi_S-1})+1\leq (H+K_S)^2$.\\
(3). Note that $(H+K_S).E=0$ for all $(-1)$-curves $E$ such that $K_S.E=-1$ and $H.E=1$. The equality now follows by recalling that
$S_1$ is obtained by blowing down every such $(-1)$-curves $E$.\\
(4). Apply the adjunction formula twice and then use Lemma \ref{adjLem}.3.\\
(5). Riemann-Roch yields that $h^0(\mathscr{O}_S(-K))=1+K_S^2+h^1(\mathscr{O}_S(-K_S))>0$. A curve $C\in |-K_S|$ has $p_a(C)=1$, such that
the adjunction formula implies that $H.C=-H.K>2.$ 
\end{proof}

We will also make use of the following result about surfaces of low degree in $\mathbb{P}^4$ due to Alexander.

\begin{thm}\label{alexList} Let $S$ be a linearly normal smooth rational surface $S$ embedded in $\mathbb{P}^4$. If the degree of $S$ is $\leq 9$ and $S$ is non-special, then the embedding complete linear system $H$ is:
\begin{center}
\begin{tabular}{l | l | l}
$\deg S$ &  $S$ & $H$\\
\hline
$3$ & $\tilde{\mathbb{P}}^2(x_1)$ & $2L-E_1$ \\
$4$ & $\tilde{\mathbb{P}}^2(x_1,\hdots , x_5)$ & $3L-E_1-\hdots -E_5$\\
$5$ & $\tilde{\mathbb{P}}^2(x_1,\hdots , x_8)$ & $4L-2E_1-E_2-\hdots- E_9$\\
$6$ & $\tilde{\mathbb{P}}^2(x_1,\hdots , x_{10})$ & $4L-E_1-\hdots -E_{10}$\\
$7$ & $\tilde{\mathbb{P}}^2(x_1,\hdots , x_{11})$ & $6L-2E_1-\hdots -2E_6-E_7-\hdots -E_{11}$\\
$8$ & $\tilde{\mathbb{P}}^2(x_1,\hdots , x_{11})$ & $7L-2E_1-\hdots -2E_{10}-E_{11}$\\
$9$ & $\tilde{\mathbb{P}}^2(x_1,\hdots , x_{10})$ & $13L-4E_1-\hdots -4E_{10}$\\
\end{tabular}
\end{center}
\end{thm}
\begin{proof}
See Theorem 1.1 in \cite{2}. 
\end{proof}

To be able to determine the configuration of the points blown up to obtain a smooth rational surface, we will be needing
the following result.

\begin{lem}\label{effLem} 
Let $H$ and $B$ be effective divisors on a smooth surface $S$ and denote $A\equiv H-B$. Suppose
$h^1(\mathscr{O}_S(H))+\chi(\mathscr{O}_S(A))>0$, suppose $h^2(\mathscr{O}_S(A))=0$ and suppose
$H.B>2p_a(B)-2$. Then $A$ is an effective divisor on $S$. 
\end{lem}
\begin{proof}

The result is clear if $\chi(\mathscr{O}_S(A))>0$, due to $h^2(\mathscr{O}_S(A))=0$. So suppose $\chi(\mathscr{O}_S(A))\leq 0$.
The assumption $H.B>2p_a(B)-2$ implies that $h^1(\mathscr{O}_B(H))=0$, by Riemann-Roch. Taking cohomology of the short exact sequence
$$0\longrightarrow\mathscr{O}_S(A)\longrightarrow\mathscr{O}_S(H)\longrightarrow\mathscr{O}_B(H)\longrightarrow 0$$
we obtain $h^1(\mathscr{O}_S(A))\geq h^1(\mathscr{O}_S(H))$. Then $h^1(\mathscr{O}_S(H))+\chi(\mathscr{O}_S(A))>0$
implies that $h^0(\mathscr{O}_S(A))>0$. \end{proof}

We are now ready to prove the converse statement in Theorem \ref{mainT}. Note that the converse statement is about $\mathbb{P}^2$ blown up in distinct points, since Theorem \ref{alexList} is about distinct points.

\begin{thm}\label{kuttThm}
Suppose there exists a linearly normal smooth rational surface $S$ of degree $11$
and sectional genus $8$ embedded in $\mathbb{P}^5$. If $i:S\hookrightarrow\mathbb{P}^5$
is an embedding and $i^*\mathscr{O}_{\mathbb{P}^5}(1)$ is the very ample line bundle associated to 
$i$, then the associated very ample divisor $H$ of $i^*\mathscr{O}_{\mathbb{P}^5}(1)$
belongs to the following divisor classes:\begin{center}
\begin{tabular}{l | l | l}
$K_S^2.$ & Type & $H$.\\
\hline
& & \\
$-10$ & $[(4,4-2e);2^1,1^{17}]$ & $ 4B+(4-2e)F-2E_1-E_2-\hdots -E_{18}$\\
$-9$ & $[(4,5-2e);2^4,1^{13}]$ & $ 4B+(5-2e)F-2E_1-\hdots -2E_4-E_5-\hdots -E_{17}$\\
$-8$ & $[(4,6-2e);2^7,1^{9}]$ & $ 4B+(6-2e)F-2E_1-\hdots -2E_7-E_8-\hdots -E_{16}$\\
$-8$ & $[7;2^7,1^{10}]$ & $ 7L-2E_1-\hdots -2E_{7}-E_8-\hdots -E_{17}$\\
$-7$ & $[9;3^6,2^2,1^8]$ & $ 9L-3E_1-\hdots -3E_6-2E_7-2E_8-E_9-\hdots -E_{16}$\\
\end{tabular}
\end{center}
\end{thm}

\begin{proof}
We proceed with the proof in two parts. First part, we use Theorem \ref{adj} and Lemma \ref{adjLem}
to produce divisor classes $H$ belongs to. Second part, we show that
all but the divisor classes stated in the theorem cannot both be very ample and special.\\
\\
\textit{Part 1 of the proof.}\\
\\
Let $S_i$ be a smooth rational surface and let $\varphi_i:S_i\hookrightarrow\mathbb{P}^{M}$ be an embedding
such that $\mathscr{O}_{S_i}(H_i)\simeq\varphi_i^*\mathscr{O}_{\mathbb{P}^{M}}(1)$. Denote $H_{i+1}$ as the adjoint
divisor of $H_i$ on $S_i$, that is $H_{i+1}:=H_i+K_i$ where $K_i:=K_{S_i}$. Furthermore, let
$\varphi_{i+1}:=\varphi_{|H_i+K_i|}$ whenever $\varphi_{i+1}$ is a birational morphism such that
$\varphi_{i+1}(S_i):=S_{i+1}$ and $\pi_i:=\pi_{H_i}$. The idea now is to use Theorem \ref{adj}
repeatedly to obtain sequences of birational morphisms
\begin{center}
\begin{tabular}{l l l l l l l l l}
$S$ & $\overset{\varphi}{\longrightarrow}$ & $S_1$ & $\overset{\varphi_1}{\longrightarrow}$ & $S_2$
& $\overset{\varphi_2}{\longrightarrow}$ & $ \cdots$ & $\overset{\varphi_{N_0-1}}{\longrightarrow}$ & $S_{N_0}$
\end{tabular}
\end{center}
where each surface $S_{i+1}$ is embedded into $\mathbb{P}^{\pi_i-1}$ due to Lemma \ref{adjLem}.1. Furthermore, note that whenever $\pi_{i} \leq 5$ the surface $S_{i+1}$ is embedded into at most a $\mathbb{P}^4$. The main idea then is that whenever $\pi_i\leq 5$ we may use Alexander's classification, Theorem \ref{alexList}, of surfaces in $\mathbb{P}^4$ or classification of surfaces in $\mathbb{P}^3$ to determine $H_i$. In which case, we may reproduce $H_0:=H\equiv H_i-K_{i-1}-K_{i-2}-\hdots- K_0$. Therefore, our strategy is to determine the minimum number of birational morphisms such that $\pi_i\leq 5$. So, we set
$$N_0:=\min\{ i\:|\: \pi_i\leq 5\}.$$ 

Note that by determining $N_0$ , invariants $(K_0^2,K_1^2,...,K_{N_0}^2)$, $\deg (S_{N_0})$ and $\pi_{N_0-1}$ we may reproduce the complete linear systems $H_0$. First of all, note that Lemma \ref{adjLem}.2 and Lemma \ref{adjLem}.5 yields that $$-11 \leq K_0^2\leq -1.$$ The rest of part 1 will be divided into cases depending on the values of $K_0^2$. \\
\\
\textit{Case 1:} $-11\leq K_0^2\leq -6$.
\\
Using Lemma \ref{adjLem}.4 we have $\pi_1=11+K_0^2$ such that $N_0=2$ if and only if $-11 \leq K_0^2\leq -6$. 
Furthermore, it follows from Lemma \ref{adjLem}.3 that $H_2^2=23+3K_0^2+K_1^2$. We subdivide this case into whether $H_2^2=(H_1+K_1)^2=0$ or not and then make use of Theorem \ref{adj}.\\
\\
\textit{Case 1.1:} Suppose $H_2^2=0$.\\
It is then straightforward,
by Lemma \ref{adjLem}.5, to check that $(K_0^2,K_1^2)$ takes the following values:
\begin{center}
\begin{tabular}{l l l l l}
$(-10,7)$ & $(-9,4)$ & $(-8,1)$ & $(-7,-2)$ & $(-6,-5)$
\end{tabular}
\end{center}
\noindent By Case 2 of Theorem \ref{adj}, $S_1$ is a Del Pezzo surface
or a conic bundle. 

If $S_1$ is a Del Pezzo surface, then the divisor class of
$H_0$ is of the following form $$H_0\equiv -K_0-K_1\equiv 6L-2E_1-\hdots -2E_{9-K_1^2}-E_{10-K_1^2}-\hdots -E_{9-K_0^2}.$$
Then Lemma \ref{binomLem}.1 implies that $\pi_0=8$ if and only if $K_1^2=7$.
So, if $S_1$ is a Del Pezzo surface then $H_0$ is of type $[6;2^2,1^{17}]$. 

If $S_1$ is a conic bundle, then $H_1\equiv -K_{1}$  so that
$$H_1\equiv 2B+aF-E_1-\hdots -E_{8-K_1^2}$$ for some $a\in\mathbb{Z}_{\geq 0}$, 
since $S_1$ has $8-K_1^2$ singular fibres and where $B^2=e$. Then it follows from $H_1.K_1=3+K_0^2$ that $2a=1-2e-K_0^2-K_1^2$. Using the latter we reproduce $H_0$ by determining $a$ for each value $(K_0^2,K_1^2)$ takes. This gives us:

\begin{center}
\begin{tabular}{l c l}
$(K_0^2,K_1^2)$ & $a$ & Type of $H_0$\\
\hline
$(-10,7)$ & $2-e$ & $[(4,4-2e)_2;2^1,1^{17}]$\\
$(-9,4)$ & $3-e$ & $[(4,5-2e)_3;2^4,1^{13}]$\\
$(-8,1)$ & $4-e$ & $[(4,6-2e)_4;2^7,1^{9}]$\\
$(-7,-2)$ & $5-e$ & $[(4,7-2e)_5;2^{10},1^{5}]$\\
$(-6,-5)$ & $6-e$ & $[(4,8-2e)_6;2^{13},1^1]$\\
\end{tabular}
\end{center}

As an example, we consider the case $(K_0^2,K_1^2)=(-7,-2)$. In this particular case, we obtain
$a=5-e$ from the relation $2a=1-2e-K_0^2-K_1^2$. So we may reproduce $H_0$.
\begin{eqnarray*}
H_0 & \equiv & H_1-K_0\\
& \equiv & 2B+(5-e)F-E_1-\hdots - E_{10}-\left(-2B+(e-2)F+E_1+\hdots +E_{15}\right)\\
& \equiv & 4B+(7-2e)F-2E_1-\hdots -2E_{10}-E_{11}-\hdots -E_{15}.
\end{eqnarray*}
\\
\noindent\textit{Case 1.2:} Suppose $H_2^2>0$.\\
Clearly $H_1$ does not belong to any of the four families stated in Case $3$ of Theorem \ref{adj} since
that would yield $\pi_0<8$ or $H_0^2<11$. Thus we may assume that $\varphi_1:S_1\longrightarrow S_2$
is a birational morphism and that $S_2$ is a smooth rational surface. Then Lemma \ref{adjLem}.2 yields the following
values for $(K_0^2,K_1^2)$.

\begin{center}
\begin{tabular}{l l l l l l l}
$(-11,K_1^2)$ & $(-8,2)$ & $(-7,1)$ & $(-7,0)$ & $(-6,0)$ & $(-6,-1)$ & $(-6,-2)$
\end{tabular}
\end{center}

Suppose $K_0^2=-11$, then $S_1$ is a surface of minimal degree in $\mathbb{P}^7$ since $\deg(S_1)=6$.
This implies that $S_1$ is a Veronese surface or a rational normal scroll.
If $S_1$ is a Veronese surface then $H_1\equiv 2L$, which contradicts $H_0^2=11$. If $S_1$ is a
rational normal scroll then $H_1\equiv B+(\alpha -e)F$, where $0\leq e<\alpha$. 
Recall that every minimal degree $d$ satisfies $d=2\alpha -e$ by Corollary
IV.2.19 in \cite{11}. An exhaustion of the pairs $(\alpha,e)$ satisfying the latter
relation yields no divisor classes $H_0$ of degree $11$. So we may rule out the case $(-11,K_1^2)$.\\ 

For the six remaining pairs of $(K_0^2,K_1^2)$ we reproduce $H_0$ by using classifications of 
smooth rational surfaces in $\mathbb{P}^n$ for $n\leq 4$. In particular, when $n=4$ then we use Theorem \ref{alexList}. Then we get the following divisor classes of $H_0$ when $S_2\subset\mathbb{P}^{\pi_1-1}$. 

\begin{center}
\begin{tabular}{l c c l l}
$(K_0^2,K_1^2)$ & $\deg(S_2)$ & $\pi_1$ & $H_2$ & Type of $H_0$\\
\hline
$(-8,2)$ & $1$ & $3$ & $L$ & $[7;2^7,1^{10}]$\\ 
$(-7,0)$ & $2$ & $4$ & $2L-E_1-E_2$ & $[8;3^2,2^7,1^7]$\\ 
$(-7,1)$ & $3$ & $4$ & $3L-E_1-\hdots-E_6$ & $[9;3^6,2^2,1^8]$\\
$(-6,-2)$ & $3$ & $5$ & $2L-E_1$ & $[8;3^1,2^{10},1^4]$\\
$(-6,-1)$ & $4$ & $5$ & $4L-2E_1-E_2-\hdots -E_8$ & $[9;3^5,2^5,1^5]$\\
$(-6,0)$ & $5$ & $5$ & $4L-E_1-\hdots -E_{10}$ & $[10;4^1,3^7,2^1,1^6]$\\
\end{tabular}
\end{center}

As an example, we consider the case $(K_0^2,K_1^2)=(-8,2)$. Note that we examined this case in Theorem \ref{constEx1} In this particular case, we have $H_2^2=23+3K_0^2+K_1^2=1$ and $\pi_1=11+K_0^2=3$. So, $S_2$ is embedded as a surface of degree $1$ in $\mathbb{P}^2$. That is, $S_2$ is the projective plane. Since the divisor $L$ is associated to the embedding of a $\mathbb{P}^2$ in $\mathbb{P}^2$, we have $H_2\equiv L$. Then we reproduce $H_0$ by reversing the adjoint process.
\begin{eqnarray*}
H_0 & \equiv & H_1-K_0\\
 & \equiv & H_2-K_1-K_0\\
 & \equiv &  L-\left(-3L+E_1+\hdots+ E_7 \right)-\left(-3L+E_1+\hdots E_{17}\right)\\
& \equiv & 7L-2E_1-\hdots -2E_7-E_8-\hdots -E_{17}.
\end{eqnarray*}

This concludes Case 1.\\

\noindent\textit{Case 2:} $-5\leq K_0^2\leq -4$.
\\
\noindent Recall that Case 1 considered the cases of $-11\leq K_0^2\leq -6$ and note that if $K_0^2\geq -5$
then Lemma \ref{adjLem}.5 applies, due to $H_1.K_1\geq -2$, such that we have $K_1^2\leq -1$. Therefore, by using Lemma \ref{adjLem}.3 and Lemma \ref{adjLem}.4, we see that $N_0=3$ occurs whenever
$-5\leq K_0^2\leq -4$.\\
\\
\textit{Case 2.1:} Suppose $H_3^2=0$.\\
Then $K_0^2\leq K_1^2\leq K_2^2$ and Lemma \ref{adjLem}.2 yields the following
possibilities for $(K_0^2,K_1^2,K_2^2)$. 

\begin{center}
\begin{tabular}{l l l l l l}
$(-5,-1,-1)$ & $(-5,-2,2)$ & $(-5,-3,5)$ & $(-5,-4,8)$ & $(-4,-3,0)$ & $(-4,-4,3)$
\end{tabular}
\end{center} 

\noindent If $S_2$ is a Del Pezzo surface, then the adjunction formula yields that $\pi_0=1+K_1^2+2K_2^2$.
So $\pi_0=8$ occurs only in the case $(-5,-3,5)$ such that $H_0$ is of type $[9;3^4,2^8,1^2]$.

If $S_2$ is a conic bundle, then $H_2\equiv 2B+aF-E_1-\hdots -E_{8-K_2^2}$
for some $a\in\mathbb{Z}_{\geq 0}$. Then $H_2.K_2=3+K_0^2+K_1^2$ gives us $2a=1-2e-\sum_1^3K_i^2$, so that
we may reproduce $H_0$.

\begin{center}
\begin{tabular}{l c l}
$(K_0^2,K_1^2,K_2^2)$ & $a$ & Type of $H_0$\\
\hline
$(-5,-1,-1)$ & $4-e$ & $[(6,8-3e)_4;3^9,1^4]$\\
$(-5,-2,2)$ & $3-e$ & $[(6,7-3e)_3;3^6,2^4,1^3]$\\
$(-5,-3,5)$ & $2-e$ & $[(6,6-3e)_2;3^3,2^8,1^2]$\\
$(-5,-4,8)$ & $1-e$ & $[(6,5-3e)_1;2^{12},1^1]$\\
$(-4,-3,0)$ & $4-e$ & $[(6,8-3e)_4;3^{8},2^3,1^1]$\\
$(-4,-4,3)$ & $3-e$ & $[(6,7-3e)_3;3^{5},2^7]$\\
& &
\end{tabular}
\end{center}
\noindent\textit{Case 2.2:} Suppose $H_3^2>0$.\\
\noindent Clearly,
$H_0$ does not belong to any of the four families of Case 3 in Theorem \ref{adj}.
So we may assume that $\varphi_3:S_2\longrightarrow S_3$ is a birational morphism.
By ruling out the cases when $1\leq \pi_2 \leq 3$, we obtain the following 
possibilities for $(K_0^2,K_1^2,K_2^2)$.\\
\\
\begin{center}
\begin{tabular}{l l l l l}
$(-5,-4,K_2^2)$ & $(-5,-1,0)$ & $(-4,-3,2)$ & $(-4,-2,-1)$ & $(-4,-2,0)$
\end{tabular}
\end{center}

\noindent Suppose $(K_0^2,K_1^2)=(-5,-4)$. Then $S_2$ is a surface of minimal degree in $\mathbb{P}^5$. 
If $S_2$ is a Veronese surface, then $H_2\equiv 2L$ which yields that $H_0$ is of type
$[8;2^{13},1^1]$. If $S_2$ is a minimal rational scroll, then $H_2\equiv B+(\alpha-e)F$ where 
$0\leq e<\alpha$ and $4=2\alpha-e$.
An investigation of the pairs $(\alpha,e)$ reveal that $H_0^2=11$ is not true.
For the remaining triples $(K_0^2,K_1^2,K_2^2)$ we
may now reconstruct $H_0$.

\begin{center}
\begin{tabular}{l c c l l}
$(K_0^2,K_1^2,K_2^2)$ & $\deg(S_3)$ & $\pi_2$ & $H_3$ & Type of $H_0$\\
\hline
$(-5,-1,0)$ & $1$ & $3$ & $L$ & $[10;3^9,2^1,1^4]$\\ 
$(-4,-3,1)$ & $1$ & $3$ & $L$ & $[10;3^8,2^5,1^1]$\\ 
$(-4,-2,-1)$ & $2$ & $4$ & $2L-E_1-E_2$ & $[11;4^2,3^8,2^1,1^2]$\\
$(-4,-2,0)$ & $3$ & $4$ & $3L-E_1-\hdots -E_6$ & $[12;4^6,3^3,2^2,1^2]$\\
\end{tabular}
\end{center}
 
This concludes Case 2.\\
\\
\noindent\textit{Case 3:} $K_0^2=-3$.\\
\noindent Then we have $-3\leq K_1^2\leq -1$ and $K_1^2\leq K_2^2$. Note that the case $(-3,-3,K_2^2)$ 
satisfies $\pi_2=5$.  Furthermore, every other case satisfies $H_2.K_2\geq -2$ such that $K_2^2\leq -1$. In particular this means that $N_0=4$ whenever $K_0^2=-3$.\\

\noindent \textit{Case 3.1:} $H_4^2=0$.\\
Then $(K_0^2,K_1^2,K_2^2,K_3^2)$ takes the following values.

\begin{center}
\begin{tabular}{l l}
$(-3,-2,-2,2)$ & $(-3,-2,-1,-1)$
\end{tabular}
\end{center}

\noindent If $S_3$ is a Del Pezzo surface, then $\pi_0=5+\sum_1^3 i\cdot K_i^2$ implies that none of the tuples
above satisfy $\pi_0=8$. If $S_3$ is a conic bundle and $H_3.F=a$, then $2a=1-2e-\sum_0^3 K_i^2$ yields the following
possibilites for $H_0$.

\begin{center}
\begin{tabular}{l c l}
$(K_0^2,K_1^2,K_2^2,K_3^3)$ & $a$ & Type of $H_0$\\
\hline
$(-3,-2,-2,2)$ & $3-e$ & $[(8,9-4e)_3;4^6,3^{4},1^1]$\\
$(-3,-2,-1,-1)$ & $4-e$ & $[(8,10-4e)_4;4^9,2^1,1^1]$\\
& &
\end{tabular}
\end{center}

\noindent \textit{Case 3.2:} $H_4^2>0$.\\
\noindent Clearly, none of the four families
in Case 3 in Theorem \ref{adj} occur. So we may assume $\varphi_4:S_3\longrightarrow S_4$ is a birational
morphism. Then we may reproduce $H_0$.

\begin{center}
\begin{tabular}{c l c c l l}
$N_0$ & $(K_0^2,\hdots, K_{N_0-1}^2)$ & $\deg(S_{N_0})$ & $\pi_{N_0-1}$ & $H_{N_0}$ & Type of $H_0$\\
\hline
$3$ & $(-3,-3,-2)$ & $3$ & $5$ & $2L-E_1$ & $[11;4^1,3^{10},2^1]$\\ 
$3$ & $(-3,-3,-1)$ & $4$ & $5$ & $3L-E_1-\hdots -E_5$ & $[12;4^5,3^5,2^2]$\\ 
$3$ & $(-3,-3,0)$ &  $5$ & $5$ & $4L-2E_1-E_2-\hdots -E_8$ & $[13;5^1,4^7,3^1,2^3]$\\
$4$ & $(-3,-2,-1,0)$ & $1$ & $3$ & $L$ & $[13;4^9,3^1,2^1,1^1]$\\
$4$ & $(-3,-2,-1,-1)$ & $5$ & $5$ & $4L-2E_1-E_2-\hdots -E_8$ & $[16;6^1,5^7,4^2,3^1,2^1]$\\
\end{tabular}
\end{center}

This concludes Case 3.\\
\\
\noindent\textit{Case 4:} $-2\leq K_0^2\leq -1$.\\
\noindent First, 
suppose $K_0^2=-2$. Then $-2\leq K_1^2\leq -1$ and $K_1^2\leq K_2^2\leq -1$.
In the case $(-2,-2,-2)$ we obtain $\pi_3=5$, that is $N_0=4$.  In the remaining cases
$H_3.K_3\geq -2$ implies that $K_2\leq K_3\leq -1$ such that $K_2^2=K_3^2=-1$.
In the case $(-2,-2,-1,-1)$ we obtain $\pi_4<5$, that is $N_0=5$. The remaining
case $(-2,-2,-1,-1)$ yields that $K_3^2\leq K_4^2\leq 0$ in which case $\pi_5<5$, that is
$N_0=6$.

Next, suppose $K_0^2=-1$. Then $H_i.K_i\geq -2$, for $1\leq i\leq 5$, which implies that
$K_0=\hdots =K_5=-1$. It then follows that $N_0=7$.\\
\\
\textit{Case 4.1:} $H_{N_0}^2=0$, whenever $5\leq N_0\leq 7$.\\
Then we obtain the following choices for $(K_0^2,\hdots ,K_{N_0}^2)$.

\begin{center}
\begin{tabular}{l l}
$(-2,-2,-1,-1,-1)$ & $(-2,-2,-1,-1,-1,8)$
\end{tabular}
\end{center}

\noindent It is clear that $S_{N_0}$ can not be a Del Pezzo surface for neither of the
tuples above. If $S_{N_0}$ is a conic bundle, where $H_{N_0-1}.F=a$, then by using
$2a=1-2e-\sum K_i^2$ we may reconstruct $H_0$.

\begin{center}
\begin{tabular}{c l c l}
$N_0$ & $(K_0^2,\hdots, K^2_{N_0})$ & $a$ & Type of $H_0$\\
\hline
$5$ & $(-2,-2,-1,-1,-1)$ & $4-e$ & $[(10,12-5e)_3;5^9,2^1]$\\
$6$ & $(-2,-2,-1,-1,-1,8)$ & $4-e$ & $[(12,10-6e)_3;5^{9},2^1]$\\
& & &
\end{tabular}
\end{center}

\noindent \textit{Case 4.2:} $H_{N_0}^2>0$, whenever $5\leq N_0\leq 7$.\\
Clearly, none of the four families
in Case $3$ in Theorem \ref{adj} occur such that we may assume $\varphi_{N_0}:S_{N_0-1}\longrightarrow S_{N_0}$
is a birational morphism. Then $(K_0^2,\hdots ,K_{N_0}^2)$ takes the following values.

\begin{center}
\begin{tabular}{l l l l l}
$(-2,-2,-2,0)$ & $(-2,-2,-2,-1)$ & $(-2,-2,-2,-2)$ & $(-2,-2,-1,-1,0)$\\
$(-2,-1,-1,-1,-1,-1)$ & $(-1,-1,-1,-1,-1,-1,0)$ & $(-1,-1,-1,-1,-1,-1,-1)$ &
\end{tabular}
\end{center}

\noindent Reconstructing $H_0$ is each of the cases above, we obtain the following.

\begin{center}
\begin{tabular}{c l c c l l}
$N_0$ & $(K_0^2,\hdots, K_{N_0-1}^2)$ & $\deg(S_{N_0})$ & $\pi_{N_0-1}$ & $H_{N_0}$ & Type of $H_0$\\
\hline
$4$ & $(-2,-2,-2,,0)$ & $3$ & $5$ & $2L-E_1$ & $[14;5^1,4^{10}]$\\ 
$4$ & $(-2,-2,-2,-1)$ & $4$ & $5$ & $3L-E_1-\hdots -E_5$ & $[15;5^5,4^5,3^1]$\\ 
$4$ & $(-2,-2,-2,-2)$ &  $5$ & $5$ & $4L-2E_1-E_2-\hdots -E_8$ & $[16;6^1,5^7,4^1,3^2]$\\
$5$ & $(-2,-2,-1,-1,0)$ & $1$ & $3$ & $L$ & $[16;5^9,4^1,2^1]$\\
$6$ & $(-2,-1,-1,-1,-1,-1)$ & $1$ & $3$ & $L$ & $[19;6^9,5^1,1^1]$\\
$7$ & $(-1,-1,-1,-1,-1,-1,0)$ & $4$ & $5$ & $3L-E_1-\hdots -E_5$ & $[24;8^5,7^5]$\\
$7$ & $(-1,-1,-1,-1,-1,-1,-1)$ & $5$ & $5$ & $4L-2E_1-E_2-\hdots -E_8$ & $[25;9^1,8^7,7^1,6^1]$\\
\end{tabular}
\end{center}

\noindent This exhausts all possibilities for $K_0^2$ and therefore concludes the first part of the proof.\\
\\
\textit{Part 2 of the proof}.\\
\\
\noindent The idea now is to show that every divisor class obtained in Part 1 of the proof, except
the divisor classes in the statement of the Theorem, cannot simultaneously be both very ample and have six global
sections on $S$. We show this by finding an explicit decomposition $H\equiv A+B$, where
$A$ will be effective by applying Lemma \ref{effLem} and the numerical invariants of $A$ will contradict
Proposition \ref{lowgenPrp}.\\

\begin{center}
\begin{tabular}{l | l | c | c | c | c | c | c}
Type of $H$. & Type of $A$. & $\chi(\mathscr{O}_S(A))$ & $p_a(A)$ & $H.A$ & $\chi(\mathscr{O}_S(B))$ & $p_a(B)$ & $H.B$\\
& & & & & & &\\
\hline
$[25;9^1,8^7,7^1,6^1]_1$ & $[9;3^8,2^2]$ & $1$ & $2$ & $4$ & $-$ & $-$ & $-$\\
$[24;8^5,7^5]$ & $[8;3^5,2^5]$ & $0$ & $1$ & $2$ & $3$ & $5$ & $9$\\
$[19;6^9,5^1,1^1]$ & $[9;3^8,2^2,1^1]$ & $0$ & $2$ & $4$ & $2$ & $3$ & $7$\\
$[16;5^9,4^1,2^1]_2$ & $[6;2^9,1^1]$ & $0$ & $1$ & $2$ & $3$ & $5$ & $9$\\
$[16;6^1,5^7,4^1,3^2]$ & $[7;3^1,2^{10}]$ & $0$ & $2$ & $4$ & $2$ & $3$ & $7$\\
$[15;5^5,4^5,3^1]$ & $[5;2^5,1^6]$ & $0$ & $1$ & $2$ & $3$ & $5$ & $9$\\
$[14;5^1,4^{10}]$ & $[7;3^1,2^{10}]$ & $0$ & $2$ & $3$ & $3$ & $4$ & $8$\\
$[(12,10-6e);5^9,2^1]_3$ & $[(2,2-e);1^9]$ & $0$ & $1$ & $2$ & $3$ & $5$ & $9$\\
$[(10,12-5e);5^9,2^1]_3$ & $[(2,2-e);1^9]$ & $0$ & $1$ & $2$ & $3$ & $5$ & $9$\\
$[16;6^1,5^7,4^2,1^2]$ & $[6;2^8,1^4]$ & $0$ & $2$ & $4$ & $2$ & $3$ & $7$\\
$[13;4^9,3^1,2^1,1^1]$ & $[6;2^8,1^4]$ & $0$ & $2$ & $4$ & $2$ & $3$ & $7$\\
$[13;5^1,4^7,3^1,2^3]$ & $[6;2^8,1^4]$ & $0$ & $2$ & $3$ & $3$ & $4$ & $8$\\
$[12;4^5,3^5,2^2]$ & $[6;2^8,1^4]$ & $0$ & $2$ & $4$ & $2$ & $3$ & $7$\\
$[11;4^1,3^{10},2^1]_{2,4}$ & $[3;1^{10}]$ & $0$ & $1$ & $2$ & $3$ & $5$ & $9$\\
$[(8,10-4e);4^9,2^1,1^1]_3$ & $[(2,2-e);1^9]$ & $0$ & $1$ & $2$ & $3$ & $5$ & $9$\\
$[(8,9-4e);4^6,3^4,1^1]_2$ & $[(2,3-e);2^1,1^9]$ & $0$ & $1$ & $2$ & $3$ & $5$ & $9$\\
$[12;4^6,3^3,2^2,1^2]$ & $[6;2^8,1^4]$ & $0$ & $2$ & $4$ & $2$ & $3$ & $7$\\
$[11;4^2,3^8,2^1,1^2]_4$ & $[4;2^1,1^{12}]$ & $0$ & $2$ & $4$ & $2$ & $3$ & $7$\\
$[10;3^8,2^5,1^1]$ & $[4;2^1,1^{13}]$ & $0$ & $2$ & $4$ & $2$ & $3$ & $7$\\
$[10;3^9,2^1,1^4]_{2,5}$ & $[1;1^2]$ & $1$ & $0$ & $4$ & $0$ & $4$ & $7$\\
$[8;2^{13},1^1]_2$ & $[4;2^1,1^{12}]$ & $0$ & $2$ & $4$ & $3$ & $3$ & $7$\\
$[(6,7-3e)_3;3^5,2^7]$ & $[(3,4-\frac{3}{2}e);2^4,1^8]$ & $0$ & $2$ & $4$ & $2$ & $3$ & $7$\\
$[(6,8-3e)_4;3^8,2^3,1^{1}]_2$ & $[(2,2-e);1^9]$ & $0$ & $1$ & $2$ & $3$ & $5$ & $9$\\
$[(6,5-3e)_1;2^{12},1^1]_2$ & $[(2,3-e);1^{12}]$ & $0$ & $2$ & $4$ & $2$ & $3$ & $7$\\
$[(6,6-3e)_2;3^3,2^8,1^{2}]_2$ & $[(2,3-e);1^{12}]$ & $0$ & $2$ & $4$ & $2$ & $3$ & $7$\\
$[(6,7-3e)_2;3^6,2^4,1^{3}]_2$ & $[(2,3-e);1^{12}]$ & $0$ & $2$ & $4$ & $2$ & $3$ & $7$\\
$[(6,8-3e)_2;3^9,1^{4}]_2$ & $[(2,3-e);1^{12}]$ & $0$ & $2$ & $4$ & $2$ & $3$ & $7$\\
$[9;3^4,2^8,1^2]_{2,4}$ & $[2;1^5]$ & $1$ & $0$ & $4$ & $0$ & $4$ & $7$\\
$[9;3^5,2^5,1^5]_2$ & $[3;1^{10}]$ & $0$ & $1$ & $2$ & $3$ & $5$ & $9$\\
$[8;3^1,2^{10},1^4]_2$ & $[4;2^1,1^{12}]$ &  $0$ & $2$ & $4$ & $2$ & $3$ & $7$\\
$[8;3^2,2^7,1^7]_{2,5}$ & $[2;1^5]$ & $1$ & $0$ & $4$ & $0$ & $4$ & $7$\\
$[(4,8-2e);2^{13},1^1]_2$ & $[(2,3-e);1^{12}]$ & $0$ & $2$ & $4$ & $2$ & $3$ & $7$\\
$[(4,7-2e);2^{10},1^5]_2$ & $[(2,3-e);1^{12}]$ & $0$ & $2$ & $4$ & $2$ & $3$ & $7$\\
\end{tabular}
\end{center}

\noindent Note that we have written subscripts on several of the divisor classes in the table above. 

Subscript $1$ means that combining $\chi(\mathscr{O}_S(A))$ with the
rationality of $S$ implies that $A$ is an effective divisor contradicting Proposition \ref{lowgenPrp}. Subscript $2$ means that the divisor $A$ has to be chosen relative
to the ordering $i\geq j$ if and only if $A.E_i\geq A.E_j$. Subscript $3$ occurs in three cases and means that $A.E_9=0$.  Subscript $4$ means that Lemma \ref{effLem} implies that $B$ is an effective divisor, since
$\chi(\mathscr{O}_S(B))=0$ and $H.A>2p_a(A)-2$. Then it follows from
$H.B>2p_a(B)-2$ and $A^2>2p_a(A)-2$ that $h^1(\mathscr{O}_S(H))=1$ is false, due to Lemma \ref{forsLem}.

 Now we illustrate how one may use the table above to obtain a contradiction for every divisor class without subscript $1$ or $4$. For instance, if $H$ is of type $[24;8^5,7^5]$ then $A$ is of type $[8;3^5,2^5]$ and $B$ is of type $[16;5^{10}]$. Since $\chi(\mathscr{O}_S(A))=0$ and $H.B>2p_a(B)-2$, Lemma \ref{effLem} yields that $A$ is effective whenever
$H$ is special. But this contradicts Proposition \ref{lowgenPrp} since $p_a(A)=2$ and $H.A\leq 2p_a(A)$. This in turn contradicts the very ampleness of $H$.\\
\\
\noindent Taking into account Theorem \ref{constEx1} and the lifting examples, namely Proposition \ref{liftEx1} and Proposition \ref{liftEx2}, this concludes the proof.  \end{proof}

\vskip10pt
\noindent\textbf{Acknowledgements.} Many of the results in this article were obtained during my master's thesis at the University of Oslo. 
I would like to thank my thesis supervisor Kristian Ranestad for interesting discussions and for continuous encouragement both during and after my studies.


\end{document}